%% file: main_V1_ACC24.tex
\documentclass[letterpaper, 10 pt, conference]{ieeeconf}  % Comment this line out
\IEEEoverridecommandlockouts                              % This command is only
\input{PackagesCommands}
\usepackage{multicol,lipsum}
\usepackage[export]{adjustbox}
\usepackage{tikz}
\usepackage{amsmath}
\usepackage{mathrsfs}
\usepackage{amssymb}
\usepackage[acronym]{glossaries}
\newacronym{nn}{NN}{Neural Network}
\newacronym{gd}{GD}{Gradient Descent}
\newacronym{fs}{FS}{Fsolve}
\newacronym{rsm}{RSM}{Random Search Method}
\newacronym{trd}{TRD}{Trust-Region-Dogleg}

\tikzstyle{rectRound} = [rectangle, rounded corners, text centered, draw=black, minimum width=2em]
\tikzstyle{rect} = [rectangle, text centered, draw=black, minimum width=2em]
\tikzstyle{clear} = [draw=none, text centered, minimum width=2em]
\tikzstyle{circ} = [circle, text centered, draw=black, minimum width = 2em]
\tikzstyle{circClear} = [circle, text centered, minimum width = 2em]
\tikzstyle{arrow} = [thick, -latex]
\makeatletter
\DeclareRobustCommand{\rvdots}{%
  \vbox{
    \baselineskip4\p@\lineskiplimit\z@
    \kern-\p@
    \hbox{.}\hbox{.}\hbox{.}
  }}
\makeatother

\usepackage[style=ieee ,sorting=none,maxbibnames=99,giveninits, doi=false, backend=biber]{biblatex}
\usepackage{graphics} % for pdf, bitmapped graphics files
\usepackage{amsmath} % assumes amsmath package installed
\usepackage{amssymb}  % assumes amsmath package installed

\usepackage[style=ieee ,sorting=none,maxbibnames=99,giveninits, doi=false]{biblatex}
\title{\LARGE \bf
Longitudinal Flight Dynamics Control Based on Feedback Linearization and Normal Canonical Form
}

\title{\LARGE \bf
Data-driven, Model-free, Coordinate-free,  Adaptive Attitude Estimation
}

\title{Multiplicative, Retrospective Cost Adaptive Attitude Estimation}
\title{Multiplicative Adaptive Attitude Estimation with Unknown Gyro Bias}
\title{Data-driven Attitude Filtering with Unknown Gyro Bias}
\title{Retrospective Cost Attitude Estimation \\ with Noisy Measurements and Unknown Gyro Bias}
\title{Learning-based Attitude Estimation \\ with Noisy Measurements and Unknown Gyro Bias}

\author{Parham Oveissi, Mohammad Mirtaba, Ankit Goel% 
\thanks{Parham Oveissi is a graduate student in the Department of Mechanical Engineering, University of Maryland, Baltimore County, 1000 Hilltop Circle, Baltimore, MD 21250. {\tt\small parhamo1@umbc.edu}}%
\thanks{Mohammad Mirtaba is a graduate student in the Department of Mechanical Engineering, University of Maryland, Baltimore County, 1000 Hilltop Circle, Baltimore, MD 21250. {\tt\small mmirtaba@umbc.edu}}%
\thanks{Ankit Goel is an Assistant Professor in the Department of Mechanical Engineering, University of Maryland, Baltimore County,1000 Hilltop Circle, Baltimore, MD 21250. {\tt\small ankgoel@umbc.edu }}%
}
\usepackage{graphicx}      % include this line if your document contains figures

\usepackage{textcomp}
\usepackage{calc}
\usepackage{mathtools}
\usepackage{xparse}
\usepackage{amssymb}
\usepackage{amsmath}
\usepackage{cases}
\usepackage{mathtools}
\usepackage{cuted}

\begin{document}

\maketitle
% \thispagestyle{empty}
% \pagestyle{empty}

%%%%%%%%%%%%%%%%%%%%%%%%%%%%%%%%%%%%%%%%%%%%%%%%%%%%%%%%%%%%%%%%%%%%%%%%%%%%%%%%
\begin{abstract}
% Attitude estimation is pivotal in fields such as aerospace engineering, robotics, computer vision, and augmented reality, offering precise orientation information crucial for navigation and control.
%
% Despite being a subset of the broader state estimation problem, attitude estimation poses unique challenges due to the complex geometric structure of attitude parameterization. 
This paper introduces a learning-based, data-driven attitude estimator, called the retrospective cost attitude estimator (RCAE), for the SO(3) attitude representation. 
RCAE is motivated by the multiplicative extended Kalman filter (MEKF). However, unlike MEKF, which requires computing a Jacobian to compute the correction signal, RCAC uses retrospective cost optimization that depends only on the measured data. 
Moreover, due to the structure of the correction signal, RCAE does not require explicit estimation of gyro bias. 
The performance of RCAE is verified and compared with MEKF through both numerical simulations and physical experiments.

% Unlike traditional methods such as the multiplicative extended Kalman filter, which rely on Jacobian-derived correction gains, RCAE leverages a learning-based approach.
% Specifically, it computes a multiplicative correction signal through retrospective cost optimization using measured data, enhancing robustness and accuracy in dynamic environments. 
%
% We validate the performance of RCAE through comprehensive numerical simulations and experimental trials, demonstrating its efficacy in handling scenarios with noisy attitude measurements and biased rate-gyro data. This innovative approach paves the way for more reliable and efficient attitude estimation, pushing the boundaries of what's possible in real-world applications.
% 

% 
\end{abstract}
% \keywords{one, two, three, four}
\textit{\bf keywords:} attitude estimation, attitude filtering, Kalman filter, adaptive estimation, learning-based estimation, data-driven attitude estimation.

%%%%%%%%%%%%%%%%%%%%%%%%%%%%%%%%%%%%%%%%%%%%%%%%%%%%%%%%%%%%%%%%%%%%%%%%%%%%%%%%
\section{Introduction}
Attitude estimation, also known as orientation estimation, is a pivotal technology widely used in applications across several domains, including aerospace, robotics, computer vision, and navigation. The primary objective of attitude estimation is to ascertain the orientation of an object in three-dimensional space relative to a reference frame. In many real-world applications, attitude estimation is a critical enabling technology. For example, it is essential in inertial navigation systems, enabling vehicles to navigate in unknown environments \cite{goel2021introduction, mirtaba2023design}. Robotic agents use attitude estimation to comprehend their spatial orientation, facilitating effective interaction with their surroundings \cite{roumeliotis1999smoother}. Furthermore, with advancements in computational technology and resources, attitude estimation is crucial for augmented reality experiences, 3D reconstruction, and object tracking \cite{xu2019multi}.

Attitude estimation is a special case of the state estimation problem in dynamic systems.
In general, state estimation techniques involve the propagation of the state estimate using a model of the dynamic system and a correction of the state estimate using a system measurement.
Kalman filter is the most well-known state estimation technique that applies to linear systems \cite{simon2006optimal, chui2017kalman}. 
Various extensions of the Kalman filter for nonlinear systems have been developed, such as the Extended Kalman Filter (EKF), Ensemble Kalman Filter (EnKF), Unscented Kalman Filter (UKF), and Particle filters \cite{sarkka2023bayesian}.

% Due to the nonlinearity of attitude dynamics, the attitude estimation problem is a nonlinear estimation problem. 
% The fundamental difficulty in the attitude estimation problem, however,
In addition to the nonlinearity of the attitude dynamics, the attitude estimation problem is further complicated by the special geometric structure of the mathematical representation of the attitude. 
The attitude of a rigid body can be parameterized by a $3\times 3$ direction cosine matrix (DCM), a four-dimensional quaternion vector, or three Euler angles, etc. \cite{kasdin2011engineering}. 
Each of these parameterizations entails some trade-offs.
The DCM, an element of $\rm SO(3),$ is unique for a given attitude but consists of nine real numbers. 
A quaternion, an element of $S^3,$ consists of four real numbers but suffers from non-uniqueness. 
Euler angles uniquely determine the attitude but suffer from singularities, known as the gimbal lock problem \cite{goel2021introduction}.
% 
% Due to the uniqueness of the DCM, we consider the attitude estimation problem with the $3\times3$ orthonormal matrices.  

Variations of the Kalman filter have been applied to the attitude estimation problem with DCM and quaternion parameterization. However, the additive step in the Kalman filter, also known as the \textit{data assimilation} step, often produces an estimate that violates the geometric structure of the representation \cite{lefferts1982kalman, crassidis2007survey}. \cite{markley2003attitude} proposed the Multiplicative Extended Kalman Filter (MEKF) with quaternion parameterization, replacing the additive state correction step with a multiplicative step, thus preserving the geometric structure of the quaternion. In \cite{markley2006attitude, chatuverdi2011rigid, de2014so}, the multiplicative technique was developed for the $3\times 3$ DCM.

In all of these techniques, the computation of the Kalman gain is computationally expensive, requiring the propagation of the corresponding covariance matrix \cite{oveissi2024novel}. To address this, we present a novel attitude estimator called the Retrospective Cost Attitude Estimator (RCAE). 
More details on MEKF can be found in  \cite{markley2003attitude, markley2006attitude, lefferts1982kalman}.

Like the MEKF, RCAE corrects the attitude estimate with a multiplicative correction, preserving the geometric structure of the attitude parameterization. However, unlike MEKF, the multiplicative correction signal in RCAE is computed using a learning-based approach involving retrospective cost optimization and measured data.
% 
% Learning-based estimation techniques have gained significant attention in recent years. For instance, in \cite{learningbased1}, a deep neural network is trained to model the measurement noise associated with sensors. In \cite{learningbased2}, the authors train a dilated convolutional network to estimate gyroscope bias and perform open-loop time integration on the corrected measurements. In \cite{learningbased3}, an LSTM neural network is trained to estimate Euler angles using raw IMU measurements as input. In \cite{learningbased4}, a neural network predicts the optimal fusion weights of a complementary filter, making the structure adaptive instead of using fixed gains for the accelerometer.
% 
RCAE leverages the retrospective cost optimization technique \cite{santillo2010adaptive} and is motivated by retrospective cost adaptive control (RCAC). RCAC is a data-driven, learning-based adaptive control technique applicable to stabilization, tracking, and disturbance rejection problems and has been applied to various engineering problems \cite{oveissi2024adaptive, goel2021experimental, chee2024performance, oveissi2023learning, poudel2023learning}. 
RCAE is computationally efficient since the attitude correction signal is scalar instead of a vector, as in MEKF. 
Furthermore, the correction signal is updated by a recursive retrospective cost optimization algorithm which is driven by measured data, eliminating the need for Jacobian computations. 
Additionally, unlike Kalman filter-based techniques, where an unknown gyro bias requires explicit estimation, RCAE can reject the effect of unknown gyro bias without modifying the algorithm.

The paper is organized as follows. 
Section \ref{sec:problem} formulates the attitude estimation problem. 
The retrospective cost attitude estimator is presented in Section \ref{sec:RCAE}.
In Section \ref{sec:simulation}, RCAE is implemented and validated in a numerical experiment as well as a physical experimental setup. 
Finally, the paper concludes in Section \ref{sec:conclusion} with a discussion of future work.

\section{Problem Formulation}
\label{sec:problem}
Let ${\rm F_A} = (\hat \imath_\rmA, \hat \jmath_\rmA, \hat k_\rmA)$ be an inertial frame, where 
$\hat \imath_\rmA, \hat \jmath_\rmA, \hat k_\rmA$ are mutually orthonormal vectors.
Let $\SB$ be a rigid body and let $\rm F_B = (\hat \imath_\rmB, \hat \jmath_\rmB, \hat k_\rmB)$ be a frame fixed to $\SB.$
% 
% Let $\rm F_A$ be an orthonormal, right-handed frames. 
Let the \textit{orientation matrix} $\SO_{\rm B/A}(t) \in \rm SO(3) \subset \BBR^{3\times 3}$ denote the attitude of $\rm F_B$ relative to $\rm F_A$ at time $t.$
Note that the orientation matrix $\SO_{\rm B/A}$ is the transpose of the direction cosine matrix. 
% Note that $\SO_{\rm B/A}(t) \in \rm SO(3).$
Letting $\omega_{\rm B/A|B}(t) \in \BBR^3$ denote that angular velocity vector of $\rm F_B$ relative $\rm F_A$ with respect to the frame $\rm F_B$ at time $t,$ it follows from the Poisson's equation \cite{kasdin2011engineering} that the orientation matrix satisfies
% 
% The angular velocity $\omega = \vect \omega_{\rm B/A} \Big|_\rmB $ and the orientation matrix $\SO_{\rm B/A}$ satisfies the Poisson's equation
\begin{align}
    \dot \SO_{\rm B/A} (t)
        =
            - \omega_{\rm B/A|B}(t)^\times \SO_{\rm B/A}(t).
    \label{eq:poissons_equation1}
\end{align}
where, for $x \in \BBR^3,$ 
\begin{align}
    x^\times 
        \isdef
            \matl 
                0 & -x_3 & x_2 \\
                x_3 & 0 & -x_1 \\
                -x_2 & x_1 & 0
            \matr.
\end{align}

% Note that 
% \begin{align}
%     \dot \SO_{\rm A/B} 
%         =
%             \SO_{\rm A/B} \omega^\times.
%     \label{eq:poissons_equation2}
% \end{align}
% we can solve \eqref{eq:poissons_equation1} as
The orientation of $\SB$ can thus be obtained by integrating the Poisson's equation \eqref{eq:poissons_equation1}, which yields
\begin{align}
    \SO_{\rm B/A}(t) 
        =
             e^{-\int_{0}^{t}  \omega_{\rm B/A|B} (t)^\times \,\rmd t}
             \SO_{\rm B/A}(0).
    \label{eq:poissons_equation_solution_cont}
\end{align}
Note that if the angular velocity vector is time-varying, the computation of the integral in \eqref{eq:poissons_equation_solution_cont} is an intractable problem. 
Several integration techniques to integrate \eqref{eq:poissons_equation1} are discussed in \cite{hairer2006geometric,betsch2009rigid,zhao2013novel,leyendecker2008variational,sabatini2005quaternion,whitmore2000closed}.
However, since the attitude dynamics, given by \eqref{eq:poissons_equation1}, lacks dissipation, if the initial orientation $\SO_{\rm B/A}(0)$ is not precisely known, then the error in the orientation matrix computed by integrating \eqref{eq:poissons_equation1} does not converge to zero.

Instead of continuous-time integration, which may be intractable, the orientation matrix can be discretely propagated, as shown below. 
If, for $\tau \in [t, t+\Delta t], $ $\omega_{\rm B/A|B} (\tau) = \omega \hat n,$ where $\omega \in \BBR$ and $\hat n \in \BBR^3,$ then
\begin{align}
    \SO_{\rm B/A}(t+\Delta t) 
        =
             e^{-\omega \Delta t \hat n ^\times }
             \SO_{\rm B/A}(t).
    \label{eq:poissons_equation_solution_DT}
\end{align}
In the case where the direction of the angular velocity vector changes slowly or the time step $\Delta t$ is sufficiently small, the orientation of $\SB$ can be computed using \eqref{eq:poissons_equation_solution_DT}.
However, the computed orientation matrix may not necessarily represent the attitude of $\SB$ due to discretization error in numerical integration. 
% 
% 
% Furthermore, measurement noise in the angular velocity vector may also contribute to the errors in the computed orientation matrix. 
% 
Furthermore, if the measurement of the angular velocity vector is noisy, then the orientation matrix, computed by integrating \eqref{eq:poissons_equation1}, is erroneous. 
% 
% Thus, naive integration of the Poisson's equation does not provide 

% \subsection{Orientation Measurement Model}

In this work, we assume that a noisy measurement of the orientation $\SO_{\rm B/A}$ is available.
The noisy orientation measurement is denoted by $\SO_{\rm B_m/A}.$
In the physical experiment, a noisy measurement of the orientation $\SO_{\rm B/A}$ is computed using the IMU data. 
The details of the orientation computation are described in Appendix I.
% \ref{appnd:AM}.
Alternatively, the orientation measurements can be obtained using vision-based sensors \cite{shabayek2012vision, thurrowgood2009vision, kessler2010vision}.

% when doing the numerical simulations and then in the experimental analysis we use an IMU and the equations \eqref{eq:k_A_IMU_Measurement}, \eqref{eq:j_A_IMU_Measurement} and \eqref{eq:i_A_IMU_Measurement} to get a noisy measurement of the orientation $\SO_{\rm B/A}$ .

% In practice, the orientation of $\SB$ can be computed using the measurement of two known vectors. 
% This work assumes that a potentially noisy measurement of the orientation of $\rm F_B$ is available. 
% 
% For example, 
% 

% 
% The following assumes that a noisy measurement of the orientation $\SO_{\rm B/A}$ is available. 
% The noisy orientation measurement is denoted by $\SO_{\rm B_m/A}.$
% 
% Thus, the problem is determining the orientation of $\SB$ using noisy measurements of the angular velocity vector and noisy measurements of the orientation matrix.  

The problem is to develop an attitude estimator that combines the orientation propagation and a noisy measurement of the orientation to construct a more accurate estimate of the attitude. 

% Assume that a measurement $\SO_{\rm B_m/A}$ of the orientation of frame $\rm F_B$ is available. 
% Let $\SO_{\rm B_m/A}$ denote this orientation.
% 
% Let $\rm F_{B_\rmm}$ denote the measurement of the orientation 
% The objective 

%\clearpage
\section{Retrospective Cost Attitude Estimator}
\label{sec:RCAE}
This section presents the data-driven retrospective cost attitude estimator (RCAE). 
The attitude estimator is motivated by the multiplicative extended Kalman filter (MEKF).

\subsection{Attitude Error}
Following the definition of the attitude error in \cite{markley2003attitude, de2014so, oveissi2024retrospective}, we define the attitude error $z\in \BBR$ between two orientation matrices $\SO_1$ and $\SO_2$ as 
\begin{align}
    z 
        \isdef 
            \tr (\SO_1 ^\rmT \SO_2 - I_3),
\end{align}
where $I_3$ is the $3\times 3$ identity matrix. 
Note that if $\SO_1 = \SO_2,$ then $z=0.$
Note that $z \in [-4,0].$

For $k = \{1, 2, \ldots, \},$
let $\SO_{\rm B/A}^k$ denote the orientation $\SO_{\rm B/A}(k \Delta_T)$ at time $ t = k \Delta_T.$
Let $\SO_{\rm \hat B/A}^k$ and $\SO_{\rm  B_m/A}^k$ denote an estimate and a measurement of the orientation $\SO_{\rm B/A}^k,$ respectively at step $k.$
Then, 
\begin{align}
    \SO_{\rm  \hat B/B_m}^k
        =
            \SO_{\rm  \hat B/A}^k 
            \SO_{\rm  A/B_m}^k
        =
            \SO_{\rm  \hat B/A}^k (\SO_{\rm  B_m/A}^k)^\rmT.
            \label{eq:O_b_bm}
\end{align}
Define the error 
\begin{align}
    z_k 
        \isdef
            % \tr ((\SO_{\rm B_m/A}^k)^\rmT \SO_{\rm \hat B/A}^k - I_3).
            \tr ( \SO_{\rm  \hat B/B_m}^k  - I_3).         
    \label{eq:z_k_def}
\end{align}

\subsection{Attitude Estimator}

The retrospective cost attitude estimator is
\begin{align}
    \SO^{k+1} _{\rm \hat B/A} 
        =
            e^{- (\omega_{k} \Delta_T + \eta_k )^\times}
            \SO^{k}_{\rm \hat B/A},
    \label{eq:RCAE}
\end{align}
where 
$\omega_k \isdef \omega_{\rm B/A|B}(k \Delta t) ,$ and the \textit{attitude correction signal} $\eta_k \in \BBR^3$ is given by
\begin{align}
    \eta_k 
        =
            u_k n_{\rm  \hat B/B_m}^k,
    \label{eq:eta_equation}
\end{align}
where  
$n_{\rm  \hat B/B_m}^k \in \BBR^3$ is the eigenaxis of the orientation matrix $\SO_{\rm  \hat B/B_m}^k,$
and 
the signal $u_k \in \BBR$ is computed using the retrospective cost optimization described in the next section. 
Note that 
\begin{align}
    n_{\hat{\rmB}/\rmB _m} ^k
        &=
            \frac{1}{2\sin{\theta^k_{\hat{\rmB}/\rm B_m}} }
            (\SO_{\rm \hat{B}/B_m}^k - \SO_{\rm B_m/\hat{B}}^k)^{-\times}.
\end{align}

\subsection{Retrospective Cost Optimization}
This section briefly reviews retrospective cost adaptive control (RCAC) and focuses on the algorithm for the attitude estimation problem. 
RCAC is described in detail in \cite{rahmanCSM2017}, and its extension to digital PID control is given in \cite{rezaPID}.
Consider a system
\begin{align}
    x_{k+1} &= f(k,x_k,u_k,w_k), 
    \label{eq:state}\\
    z_k     &= g(k,x_k,w_k), 
    \label{eq:output}
\end{align}
where
$x_k $ is the state,
$u_k $ is the input, 
$w_k $ is the exogenous signal that can represent commands, external disturbance, or both, and
$z_k $ is the performance variable.
The functions $f$ and $g$ represent the dynamics and output maps.
The goal is to develop an adaptive law that drives the performance variable $z_k$ to zero, asymptotically to zero without explicit knowledge of $f$ and $g.$
% \begin{align}
%     z_k \isdef r_k - y_k,
% \end{align}
% where $r_k \in \BBR^{l_y}$ is the desired output, asymptotically to zero without explicit knowledge of $f$ and $g.$
% $y_k$ to desired values with limited modeling information about \eqref{eq:state}, \eqref{eq:output}.
% % 
% Note that explicit knowledge of $f$ and $g$ is not required since RCAC requires only the input and output measurements. 

% The algorithm in \cite{rezaPID} is specialized for a PID control structure as shown in \cite{goel2020adaptive}.
% 
% Consider the control law
% \begin{align}
%     u_k 
%         =
%             \Phi_k \theta_k,
%     \label{eq:uk_reg}
% \end{align}
% where, for all $k\ge0$, 
% the regressor $\Phi_k \in \BBR^{l_u \times l_\theta}$ contains the measurements and $l_\theta$ depends on the structure of the controller.
% % 
% The controller coefficients $\theta_k \in \BBR^{l_\theta}$ are optimized by RCAC as described below. 

Consider an adaptive law
\begin{align}
    u_k = \Phi_k \theta_k,
    \label{eq:control_law}
\end{align}
where $\Phi_k \in \BBR^{l_u \times l_\theta}$ is the regressor matrix that is constructed using measurements and 
$\theta_k \in \BBR^{l_\theta}$ is the vector of the gains optimized by RCAC at step $k.$
For example, a discrete-time adaptive PID update law can be written in the  form given by \eqref{eq:control_law}, where, at step $k,$
\begin{align}
    \Phi_k 
        \isdef 
            \matl
                z_k & \gamma_k & z_k-z_{k-1}
            \matr, \quad
    \theta_k 
        =
            \matl
                K_{\rmp,k}  \\
                K_{\rmi,k} \\
                K_{\rmd,k}
            \matr, 
\end{align}
$\gamma_k = \sum_i z_i$ is the accumulated error, and 
$K_{\rmp,k}, K_{\rmi,k},$ and $K_{\rmd,k}$ are the proportional, integral, and derivative gains, respectively.
Various MIMO parameterizations of the adaptive law \eqref{eq:control_law} are described in \cite{goel_2020_sparse_para}.
To determine the gains $\theta_k$, let $\theta \in \BBR^{l_\theta}$, define the \textit{retrospective performance variable} by
%%\vspace{-5pt}
\begin{align}
    \hat{z}_{k}(\theta)
        \isdef
            z_k + 
            G_\rmf(\shiftq) (\Phi_{k} \theta - u_{k}),
    \label{eq:zhat_def}
\end{align}
where 
\begin{align}
    G_\rmf(\shiftq) 
        \isdef
            \sum_{i=1}^{n_\rmf} \frac{N_i}{\shiftq^i}
\end{align}
is an FIR filter. 
Note that $N_i \in \BBR^{l_z \times l_u}.$
% The sign of $\sigma$ is the sign of the leading numerator coefficient of the transfer function from $u_k$ to $z_k.$
%
Furthermore, define the \textit{retrospective cost function} $J_k \colon \BBR^{l_\theta} \to [0,\infty)$ by
%%\vspace{-10pt}
\begin{align}
    J_k(\theta) 
        &\isdef
            \sum_{i=0}^k
                \hat{z}_{i}(\theta) ^\rmT 
                R_z
                \hat{z}_{i}(\theta)
                 +
                (\Phi_k \theta)^\rmT
                R_u
                (\Phi_k \theta)
                \nn \\ &\quad 
                +                
                (\theta-\theta_0)^\rmT 
                P_0^{-1}
                (\theta-\theta_0),
    \label{eq:RetCost_def}
\end{align}
where $R_z \in \BBR^{l_z \times l_z}$, $ R_u \in \BBR^{l_u \times l_u},$ and $P_0\in\BBR^{l_\theta\times l_\theta}$ are positive definite; and $\theta_0\in\BBR^{l_\theta}$ is the initial vector of controller gains.
%
% For all examples in this paper, we set $\theta_0 = 0$; however, $\theta_0$ can be initialized to nonzero gains in practice if desired.

\begin{proposition}
    Consider \eqref{eq:RetCost_def}, 
    where $\theta_0 \in \BBR^{l_\theta}$ and $P_0 \in \BBR^{l_\theta \times l_\theta}$ is positive definite. 
    For all $k\ge0$, denote the minimizer of $J_k$ given by \eqref{eq:RetCost_def} by
    \begin{align}
        \theta_{k+1}
            \isdef
                \underset{ \theta \in \BBR^n  }{\operatorname{argmin}} \
                J_k({\theta}).
        \label{eq:theta_minimizer_def}
    \end{align}
    Then, for all $k\ge0$, $\theta_{k+1}$ is given by 
    \begin{align}
        \theta_{k+1} 
            &=
                \theta_k  - 
                 P_{k+1}\Phi_{\rmf, k}^\rmT R_z
                \left( z_k + \Phi_{\rmf,k} \theta_k - u_{\rmf,k} \right)
                 \nn  \\ &\quad 
                 - 
                 P_{k+1}\Phi_{k}^\rmT
                 R_u \Phi_{k} \theta_k 
                 , \label{eq:theta_update}
    \end{align}
    where 
    \begin{align}
        P_{k+1} 
            &=
                P_{k}
                -  
                P_k  
                \overline \Phi_k ^\rmT 
                \left( 
                    \overline R ^{-1} +  
                    \overline \Phi_k
                    P_k
                    \overline \Phi_k ^\rmT 
                \right)^{-1}
                \overline \Phi_k P_k,
        \label{eq:P_update_noInverse}
    \end{align}
    and
    \begin{align}
        \Phi_{\rmf,k}
            &\isdef 
                G_\rmf(\shiftq) \Phi_{k}, \quad 
        u_{\rmf,k}
            \isdef 
                G_\rmf(\shiftq) u_{k}, \quad 
        \\
        \overline \Phi_k
            &\isdef 
            \matl
                \Phi_{\rmf,k} \\
                \Phi_k
            \matr,
        \quad 
        \overline R
            \isdef 
                \matl
                    R_z & 0 \\
                    0   & R_u
                \matr.
    \end{align}
    % \begin{align}
    %     P_{k+1} 
    %         &=
    %             P_{k}
    %             -  \frac
    %                 { P_{k}\Phi_{k-1}^\rmT  \Phi_{k-1} P_{k} }
    %                 { 1 +   \Phi_{k-1} P_{k} \Phi_{k-1}^\rmT  }.
    %     \label{eq:P_update_noInverse}
    % \end{align}
\end{proposition}
% \textbf{Proof:} 
\begin{proof}
See \cite{AseemRLS,goel2020recursive}.
\end{proof}

Finally, the adaptive signal at step $k+1$ is given by
\begin{align}
    u_{k+1} = \Phi_{k+1} \theta_{k+1}.
    \label{eq:RCAE_u_update}
\end{align}

In the attitude estimation problem, 
the performance variable $z_k$ is given by \eqref{eq:z_k_def}.
% It follows from \eqref{eq:eta_equation} that the adaptive signal $u_k$ is a scalar. 
Note that both $z_k$ and the adaptive signal $u_k$ are scalars.
The RCAE is thus
\eqref{eq:RCAE}, 
\eqref{eq:eta_equation},
\eqref{eq:control_law}, 
\eqref{eq:theta_update}, and 
\eqref{eq:P_update_noInverse}, and its architecture is shown in Figure \ref{fig:RCAE_architecture}.

\begin{figure}[h]
    
    \begin{tikzpicture}[auto, node distance=2.5cm,>=latex' ]
    % Nodes
    \draw [draw=black, fill=green!20] (-4.2,-5.5)  node [xshift=12em,yshift=1em] {\textbf{Retrospective Cost Attitude Estimator}} rectangle (4.2,-1.2);
    \node at (0,0) [block, align=center] (physicalsystem) {Rigid Body} ;
    
    \node [input, below of=physicalsystem, name=input_corect, node distance = 1cm] {};
    
    \node [output, right of=physicalsystem, name=output_system, node distance = 3cm] {};

    \node [block, below of=physicalsystem, name=propagate, node distance = 2cm, align=center] {Attitude\\Estimator\\ \eqref{eq:RCAE}};

    \node [block, right of=propagate, name=error, node distance = 3cm, align=center] {Error\\computation\\ \eqref{eq:O_b_bm}, \eqref{eq:z_k_def}};

    % \node [block, left of=propagate, name=corectsum, node distance = 3cm, align=center] {Corrected\\$\omega$};
    
    \node [block, below of=propagate, name=rcae, node distance = 2cm, align=center] {Adaptive\\Signal\\ \eqref{eq:RCAE_u_update}};
    
    \node [block, left of=propagate, name=eta, node distance = 3cm, align=center] {Attitude\\Correction \\ \eqref{eq:eta_equation}};

        \draw[->] (physicalsystem) -| node [xshift=0em] {$\SO_{\rm B_m/A}^k$} (error.90);

        \draw[->] (physicalsystem) -- node [xshift=0em] {$\omega_{\rmm,k}$} (propagate);
        
        % \draw[-] (output_system) |- node {} (input_corect);
        % \draw [->] (output_system) -- node [near end]{$O$} (error);
        \draw [->] (error.270) |- node {$z_k$} (rcae.0);
        \draw [->] (rcae.180) -| node {$u_k$} (eta.270);
        \draw [->] (eta) -- node [yshift=0em] {$\eta_k$} (propagate);
        % \draw [->] (corectsum) -- node {} (propagate);
        \draw [->] (propagate) -- node {$\SO_{\rm \hat B/A}^k$} (error);
        % \draw [->] (input_corect) -| node [near end] {$\omega$} (corectsum);
    \end{tikzpicture}
    \caption{Retrospective cost attitude estimator architecture.}
    \label{fig:RCAE_architecture}
\end{figure}
%\clearpage

%\clearpage
\section{Experimental Verification and Performance Comparison}
\label{sec:simulation}
In this section, we apply the RCAE developed in the previous section to estimate the attitude of a rigid body with a noisy angular velocity and attitude measurement in a numerical simulation as well as a physical experiment. 
Furthermore, the performance of RCAE is also compared with MEKF

% \clearpage
\subsection{Numerical Simulation}
% 
% the performance of the RCAE numerically. 
% We first consider an attitude estimate problem.
The rigid body $\SB$ is assumed to be rotating with a time-varying angular velocity vector given by
\begin{align}
    \omega_{\rm B/A|B} (t)
        =
            \matl 
                80 \cos (5.0 t) \\
                60 \cos (7.0 t) \\
                40 \cos (9.0t)
            \matr.
\end{align}
% The attitude of $\SB$ at $t=0$  $\SO_{\rm B/A}(0)$
The 3-2-1 Euler angles $\psi, \theta, $ and $ \phi$ corresponding to the orientation matrix of $\SB$ at $t=0$, in degrees, are assumed to be 
\begin{align}
    \psi(0) &= 30, \
    \theta(0) = 20, \
    \phi(0) = 10.
\end{align}
The initial orientation is thus
\begin{align}
    \SO_{\rm B/A}(0) = \SO_1(10) \SO_2(20) \SO_3(30),
\end{align}
where $\SO_{1}(\phi), \SO_2(\theta),$ and $\SO_3(\psi)$ are the Euler matrices corresponding to the rotation by $\phi,$ $\theta,$ and $\psi$ degrees about the first, second, and third Euler axis, respectively.  
In this work, we simulate the orientation $\SO_{\rm B/A}(t)$ by discretely propagating  \eqref{eq:poissons_equation_solution_DT}.

The measured angular velocity vector is assumed to be
\begin{align}
    \omega_{\rmm, k} 
        = 
            \omega_k
            +
            b
            +
            w_k, 
\end{align}
where 
$\omega_k \isdef \omega_{\rm B/A|B}(k \Delta t) ,$
$b \in \BBR^3$ is an unknown bias and $w_k \sim \SN(0, \sigma_w^2 I_3)$ is a zero-mean Gaussian noise. 
In this example, we set $b = \matl 5 & 7 & 4 \matr^\rmT$ and $\sigma_w = 2$ deg/sec.
Figure \ref{fig:True_Measured_Angular_Velocity_simulation} shows the true and measured angular velocity. 

\begin{figure}[H]
    \centering
    \includegraphics[width=\columnwidth]{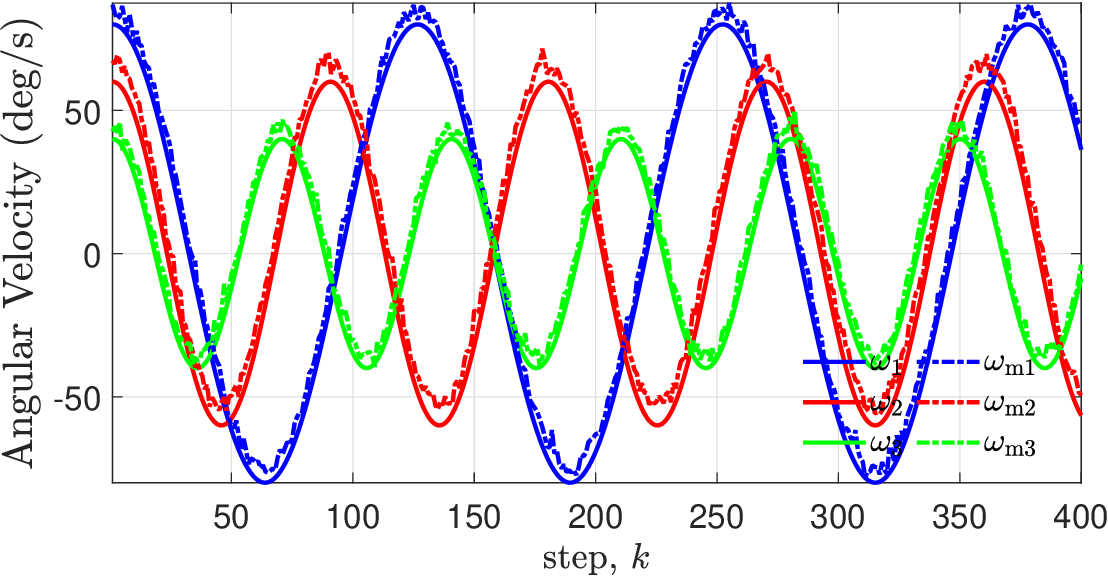}
    \caption{True and measured angular velocities.}
    \label{fig:True_Measured_Angular_Velocity_simulation}
\end{figure}

The orientation measurements are constructed as follows. 
Let $\phi_k,$ $\theta_k,$ and $\psi_k$ denote the 3-2-1 Euler angles of  $\SO_{\rm B/A}(k \Delta t).$
The measured orientation is then given by
\begin{align}
    \SO_{\rm B_m/A}^k 
        =
            \SO_1(\phi_{\rmm,k})
            \SO_2(\theta_{\rmm,k})
            \SO_3(\psi_{\rmm,k}),
\end{align}
where  
\begin{align}
    \matl 
        \phi_{\rmm, k} \\
        \theta_{\rmm, k} \\
        \psi_{\rmm, k}
    \matr
        =
            \matl 
                \phi_k \\
                \theta_k \\
                \psi_k
            \matr
            +
            v_k, 
\end{align}
and $v_k \sim \SN(0, \sigma_v^2 I_3)$ is a zero-mean Gaussian noise. 
In this example, we set $\sigma_v = 5$ deg.

In RCAE, we set 
$N_1 = 1,$
$P_0 = 0.1I_{3},$
$\lambda = 1$.
The initial attitude estimate $\SO_{\rm \hat B/A }^0 = I_{3}$ since no prior information about the attitude is assumed. 
Figure  \ref{fig:Euler_Angles_RCAF_simulation} shows the 3-2-1 Euler angles corresponding to the 
true orientation $\SO_{\rm B/A}^k,$ measured orientation $\SO_{\rm B_m/A}^k,$ and estimation orientation $\SO_{\rm \hat B/A}^k.$ 
% Note that the estimates are closer to the true values of the Euler angles than the measurements.
Figure \ref{fig:RCAF_RCAC_signals_simulation} shows 
a) the signal $u_k$ computed by RCAE, 
and 
b) the estimator gain $\theta_k$ computed by RCAE. 
Figure \ref{fig:Orientation_Animation_RCAF_simulation} shows the true and estimated frames corresponding to the true orientation $\SO_{\rm B/A}^k,$ and estimation orientation $\SO_{\rm \hat B/A}^k$ at several iteration steps during the estimation.
\begin{figure}[h]
    \centering
    \includegraphics[width=\columnwidth]{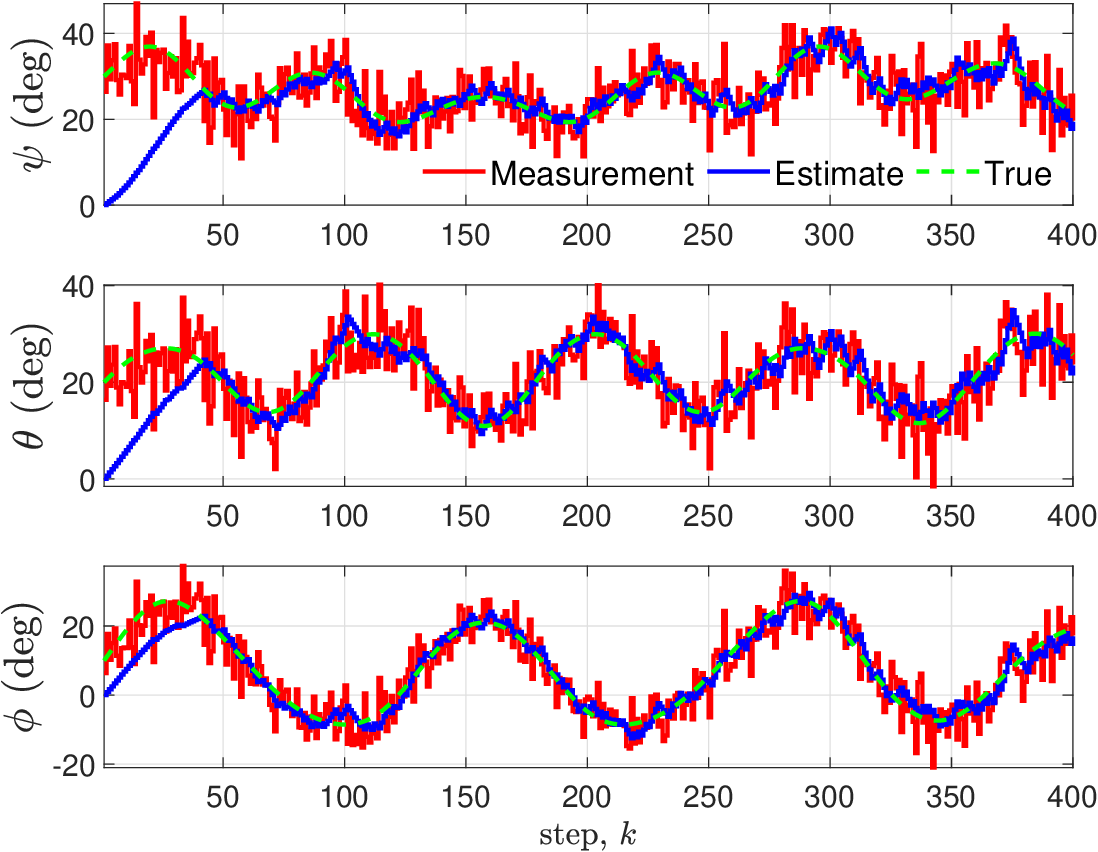}
    \caption{3-2-1 Euler angles corresponding to the 
    true orientation $\SO_{\rm B/A}^k,$
    measured orientation $\SO_{\rm B_m/A}^k,$ and 
    estimated orientation $\SO_{\rm \hat B/A}^k$ using RCAE.
    % Note that the estimates are closer to the true values of the Euler angles than the measurements. 
    }
    \label{fig:Euler_Angles_RCAF_simulation}
\end{figure}
% Figure \ref{fig:True_Measured_Angular_Velocity} shows the true and measured angular velocities.

\begin{figure}[H]
    \centering
    \includegraphics[width=\columnwidth]{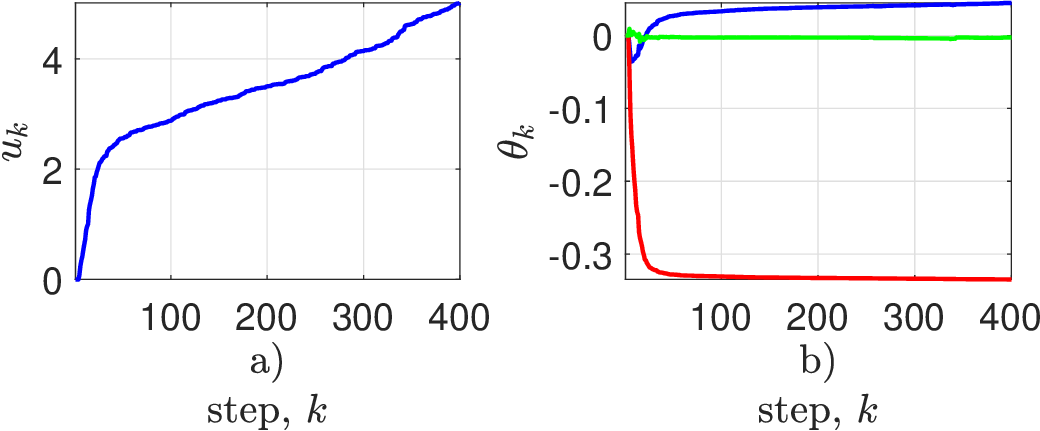}
    \caption{
    Retrospective cost attitude estimation.  
    a) shows the signal $u_k$ computed by RCAE, 
    and 
    b) shows the estimator gain $\theta_k$ optimized by RCAE. 
    }
    \label{fig:RCAF_RCAC_signals_simulation}
\end{figure}

\begin{figure}[h]
    \centering
    \includegraphics[width=\columnwidth]{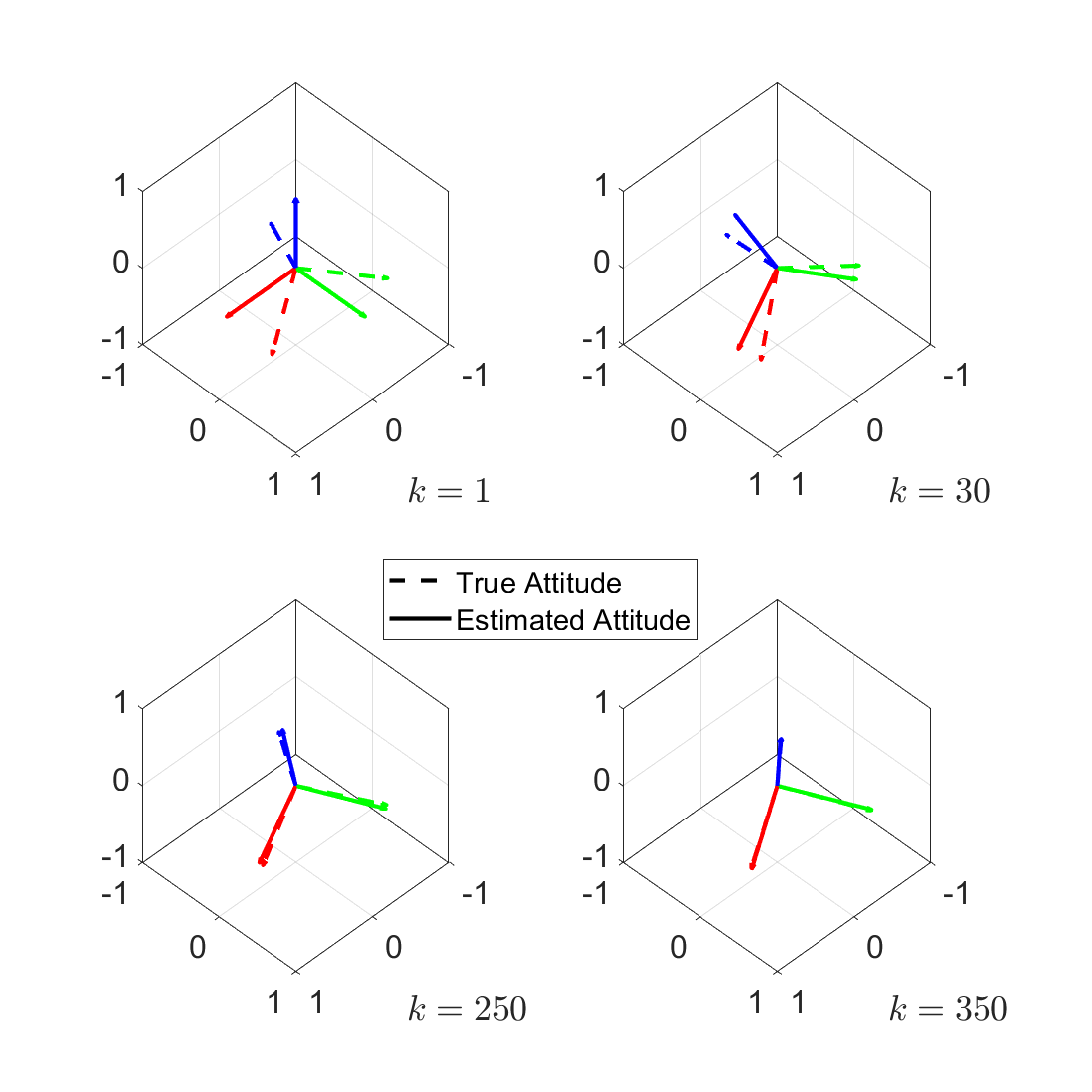}
    \caption{True and estimated frames corresponding to the
    true orientation $\SO_{\rm B/A}^k$ and 
    estimated orientation $\SO_{\rm \hat B/A}^k.$}
    \label{fig:Orientation_Animation_RCAF_simulation}
\end{figure}
 
Next, the attitude is estimated using MEKF, where the attitude is represented as a quaternion. 
In MEKF, we set the initial covariance of quaternion error $P(0) = 10^4 I_6, $
the process covariance $Q = \begin{bmatrix}
    0.0001 I_3 &0\\
    0 & I_3
\end{bmatrix}$ and the measurement covariance $R = \begin{bmatrix}
    0.01I_3 &0\\
    0 & 100I_3
\end{bmatrix}.$
Since MEKF provides the estimate of the attitude in the quaternion form, we convert the attitude estimate to the orientation matrix and the corresponding 3-2-1 Euler angles. 
In particular, the orientation matrix $\SO_{\rm B/A}$ corresponding to the quaternion 
\begin{align}
    q_{\rm B/A} 
        =
            \matl    
                \eta_{\rm B/A} \\
                \varepsilon_{\rm B/A}
            \matr, 
\end{align}
where $\eta_{\rm B/A} \in [-1, 1]$ and $\varepsilon_{\rm B/A} = \BBR^3$
is
\begin{align}
    \SO_{\rm B/A}
        =
            I_3 - 2 \eta_{\rm B/A} \varepsilon_{\rm B/A}^\times 
            + 
            2 \varepsilon_{\rm B/A}^{\times 2}.
\end{align}
% \begin{align}
%     \SO_{\rm B/A}
%         = \matl 
%         2b^2 -1 & 2bc & 2bd \\[1ex]
%         2bc & 2c^2 -1 & 2cd \\[1ex]
%         2bd & 2cd & 2d^2 -1
%         \matr ^{-1}
% \end{align}
% 
Figure \ref{fig:RCAF_MEKF_Performance_simulation} shows the attitude error $z_k$ computed with both MEKF and RCAE. 
Figure \ref{fig:Euler_Angles_RCAF_MEKF_Errors_simulation} shows the absolute value of the 3-2-1 Euler angle errors $e_\psi, e_{\theta}, $ and $e_\phi$ obtained with MEKF and RCAE. 
Note that, unlike the MEKF, the RCAE directly estimates the orientation matrix and is computationally less expensive. Additionally, while the MEKF requires explicit estimation and correction of gyro bias, the RCAE can compensate for an unknown constant bias without needing to estimate it.
\begin{figure}[h!]
    \centering
    \includegraphics[width=\columnwidth]{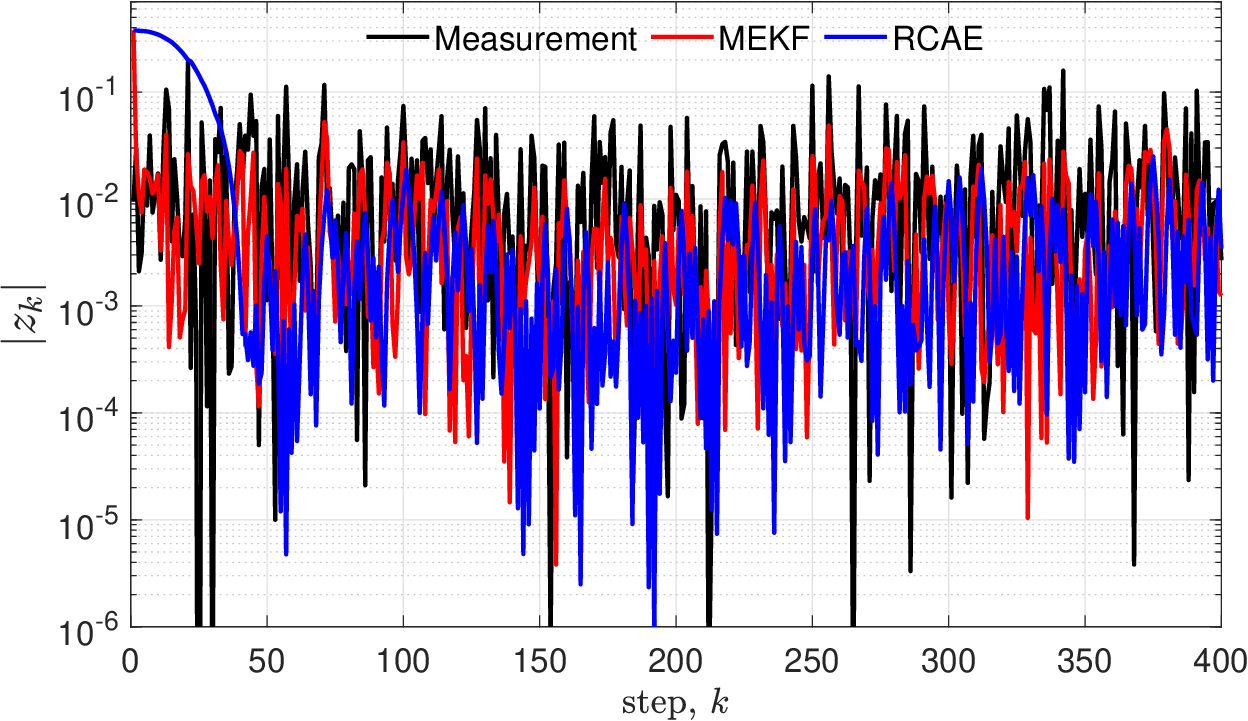}
    \caption{
    Attitude error obtained with RCAE and MEKF. 
    }
    \label{fig:RCAF_MEKF_Performance_simulation}
\end{figure}

\begin{figure}[h!]
    \centering
    \includegraphics[width=\columnwidth]{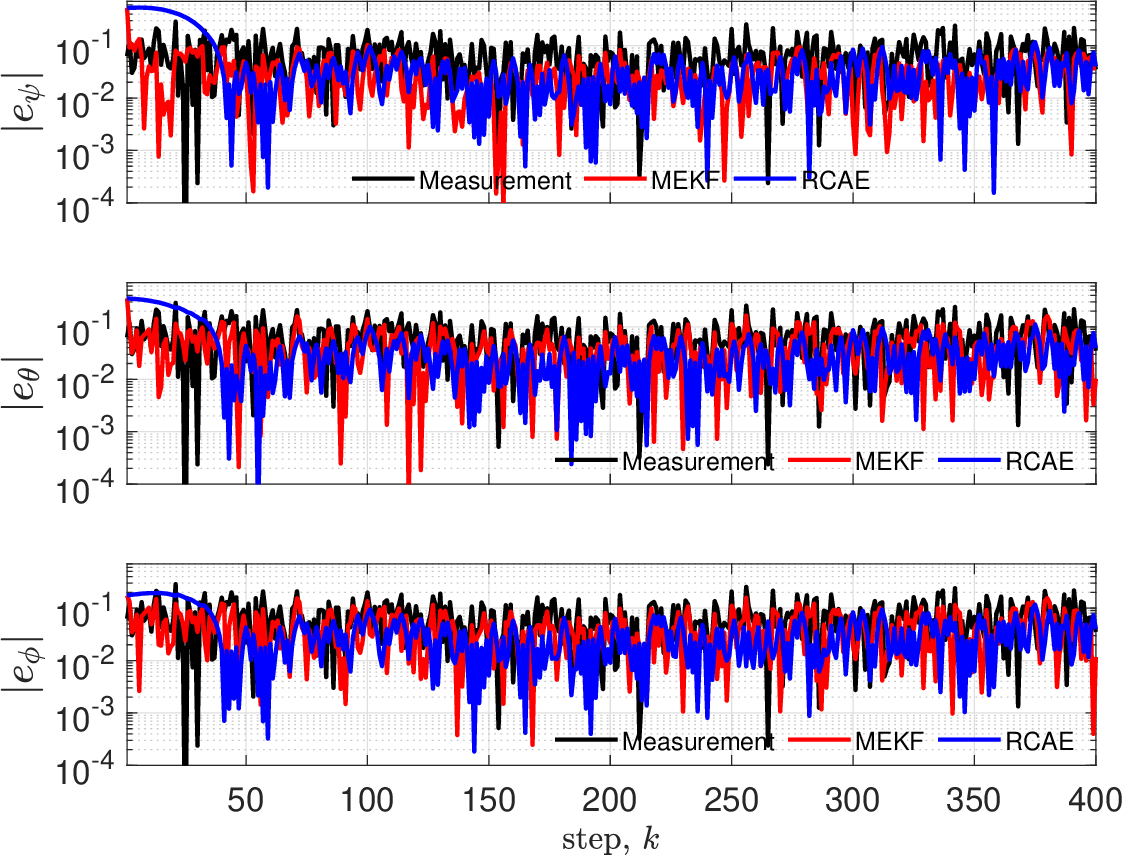}
    \caption{
    3-2-1 Euler angle errors obtained with RCAE and MEKF. 
    }
    \label{fig:Euler_Angles_RCAF_MEKF_Errors_simulation}
\end{figure}

% Next, the same numerical experiment is conducted using MEKF to estimate the attitude. For the computation of the attitude of a body using the MEKF algorithm, the initial data required are $\hat{q}(0)$, initial attitude; $\rm P(0)$, the initial covariance of the quaternion error vector $a$; $\rm Q$, the covariance of the gyroscopes; and $\rm{R}$, the covariance of each of an external vector measurement. 

% Figure  \ref{fig:Euler_Angles_MEKF_simulation} shows the 3-2-1 Euler angles corresponding to the 
% true orientation $\SO_{\rm B/A}^k,$ measured orientation $\SO_{\rm B_m/A}^k,$ and estimation orientation $\SO_{\rm \hat B/A}^k.$ 
% Note that the estimates are closer to the true values of the Euler angles than the measurements.

% \begin{figure}[h!]
%     \centering
%     \includegraphics[width=\columnwidth]{Figures/Euler_Angles_MEKF_simulation.eps}
%     \caption{3-2-1 Euler angles corresponding to the 
%     true orientation $\SO_{\rm B/A}^k,$
%     measured orientation $\SO_{\rm B_m/A}^k,$ and 
%     estimation orientation $\SO_{\rm \hat B/A}^k.$
%     Note that the estimates are closer to the true values of the Euler angles than the measurements.}
%     \label{fig:Euler_Angles_MEKF_simulation}
% \end{figure}

% \clearpage
\subsection{Physical Experiment}
\label{sec:experiments}
To investigate and verify the performance of the RCAE algorithm, an experimental setup featuring a \href{https://cdn-shop.adafruit.com/datasheets/BST_BNO055_DS000_12.pdf}{BNO055} sensor is used. The BNO055 is a popular 9-DOF IMU module that integrates a triaxial accelerometer, gyroscope, and magnetometer. It also includes a microcontroller with sensor fusion algorithms, allowing it to provide directly calibrated orientation data (quaternions, Euler angles). In this setup, the raw data from the accelerometer and magnetometer are used to obtain a noisy orientation measurement. 
In contrast, the orientation data from the BNO055's built-in sensor fusion algorithms serves as the ground truth for performance comparison. 

First, the performance of the RCAE algorithm is evaluated using the experimental setup. The gyroscope data is utilized in equation \eqref{eq:RCAE} to propagate Poisson's equation, while the raw accelerometer and magnetometer data are used to construct a noisy measurement of $\SO_{\rm B/A}$, following equations \eqref{eq:k_A_IMU_Measurement}, \eqref{eq:j_A_IMU_Measurement} and \eqref{eq:i_A_IMU_Measurement}.

Figure  \ref{fig:Euler_Angles_RCAF_experiment} shows the 3-2-1 Euler angles corresponding to the 
true orientation $\SO_{\rm B/A}^k,$ measured orientation $\SO_{\rm B_m/A}^k,$ and estimation orientation $\SO_{\rm \hat B/A}^k.$ 
\begin{figure}[h!]
    \centering
    \includegraphics[width=\columnwidth]{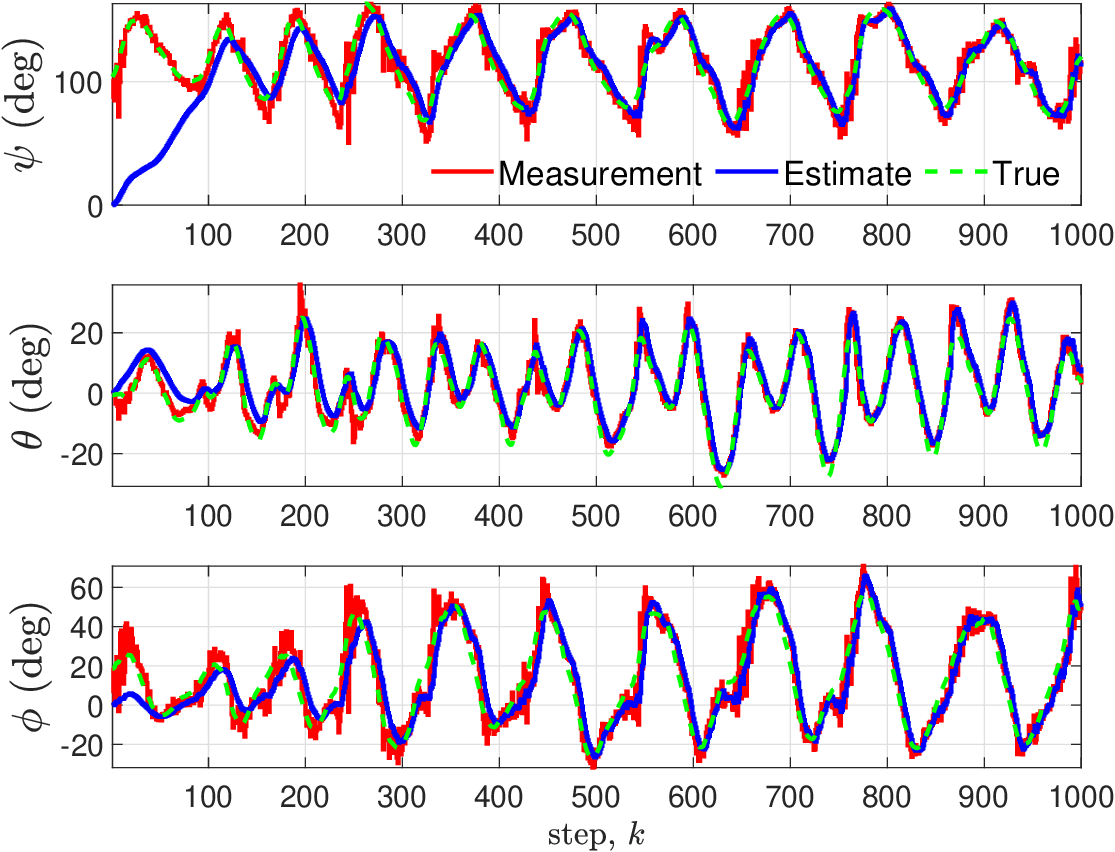}
    \caption{3-2-1 Euler angles corresponding to the 
    true orientation $\SO_{\rm B/A}^k,$
    measured orientation $\SO_{\rm B_m/A}^k,$ and 
    estimation orientation $\SO_{\rm \hat B/A}^k.$
    Note that the estimates are closer to the true values of the Euler angles than the measurements.}
    \label{fig:Euler_Angles_RCAF_experiment}
\end{figure}

Figure \ref{fig:Orientation_Animation_RCAF_experiment} shows the true and estimated frames corresponding to the true orientation $\SO_{\rm B/A}^k,$ and estimation orientation $\SO_{\rm \hat B/A}^k$ at several iteration steps during the estimation.
\begin{figure}[h!]
    \centering
    \includegraphics[width=\columnwidth]{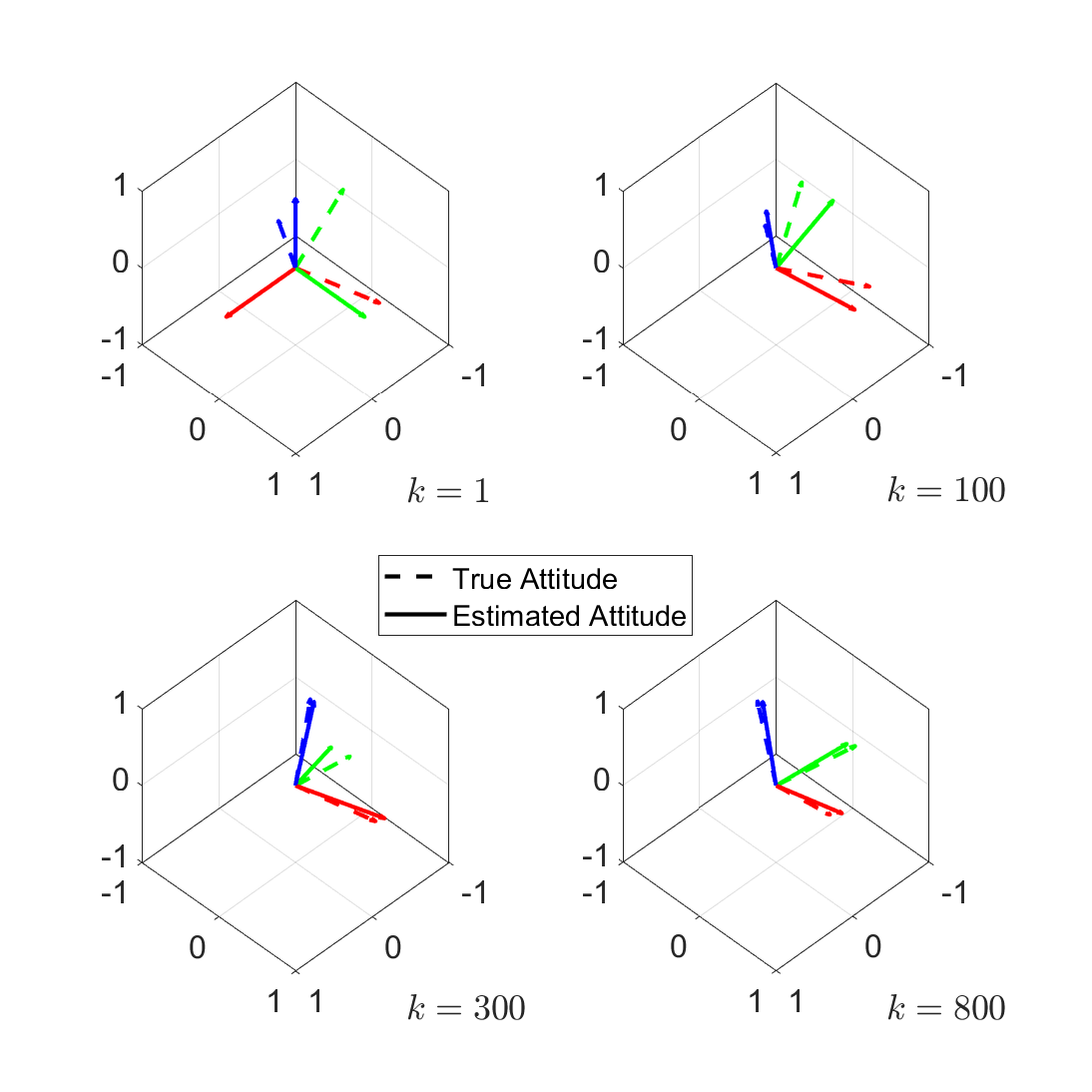}
    \caption{True and estimated frames corresponding to the
    true orientation $\SO_{\rm B/A}^k$ and 
    estimation orientation $\SO_{\rm \hat B/A}^k.$}
    \label{fig:Orientation_Animation_RCAF_experiment}
\end{figure}

Next, the attitude is estimated using MEKF, where the attitude is represented as a quaternion. 
In MEKF, we set the initial covariance of quaternion error $P(0) = 10^4 I_6, $
the process covariance $Q = \begin{bmatrix}
    0.0001 I_3 &0\\
    0 & I_3
\end{bmatrix}$ and the measurement covariance $R = \begin{bmatrix}
    0.01I_3 &0\\
    0 & 100I_3
\end{bmatrix}.$ A reference gravity vector and magnetic field are also needed, obtained from a fixed accelerometer and the World Magnetic Model (WMM), respectively.

Figure \ref{fig:RCAF_MEKF_Performance_experiment} shows the attitude error $z_k$ computed with both MEKF and RCAE. 
Figure \ref{fig:Euler_Angles_RCAF_MEKF_Errors_experiment} shows the absolute value of the 3-2-1 Euler angle errors $e_\psi, e_{\theta}, $ and $e_\phi$ obtained with MEKF and RCAE. 
Note that, unlike the MEKF, the RCAE directly estimates the orientation matrix and is computationally less expensive. Additionally, while the MEKF requires explicit estimation and correction of gyro bias, the RCAE can compensate for an unknown constant bias without needing to estimate it.

\begin{figure}[h!]
    \centering
    \includegraphics[width=\columnwidth]{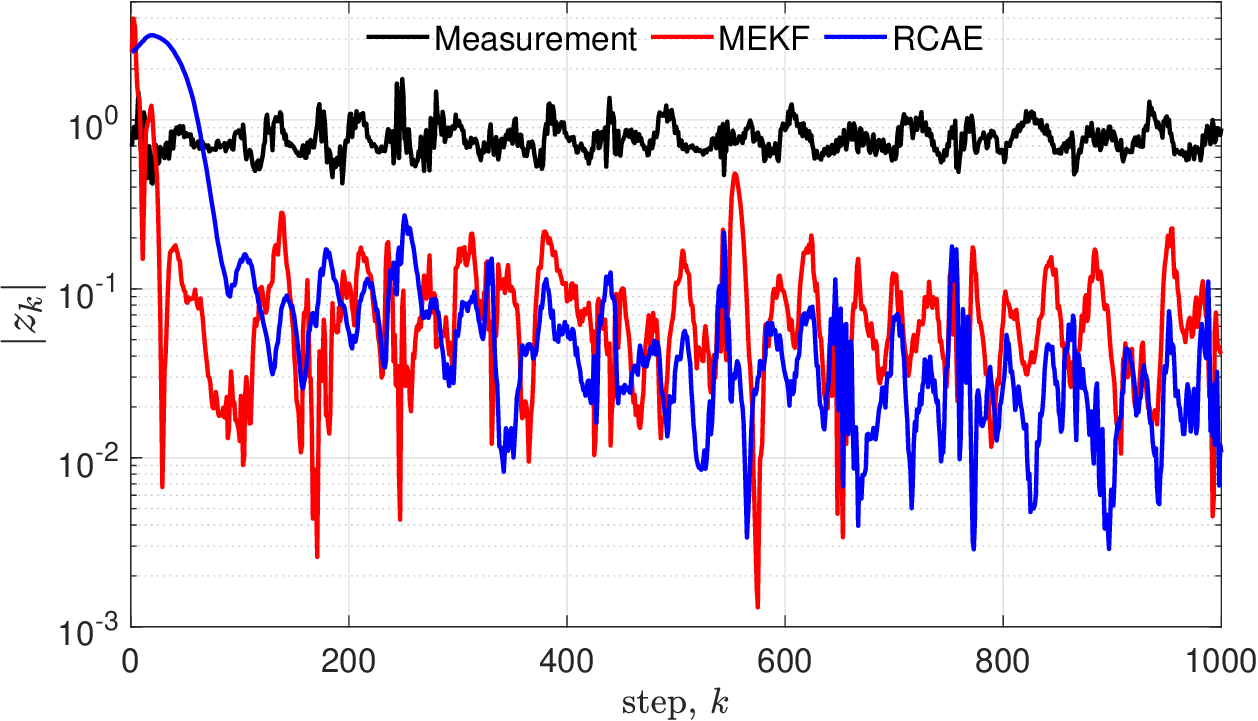}
    \caption{
    Attitude error obtained with RCAE and MEKF. 
    }
    \label{fig:RCAF_MEKF_Performance_experiment}
\end{figure}

\begin{figure}[h!]
    \centering
    \includegraphics[width=\columnwidth]{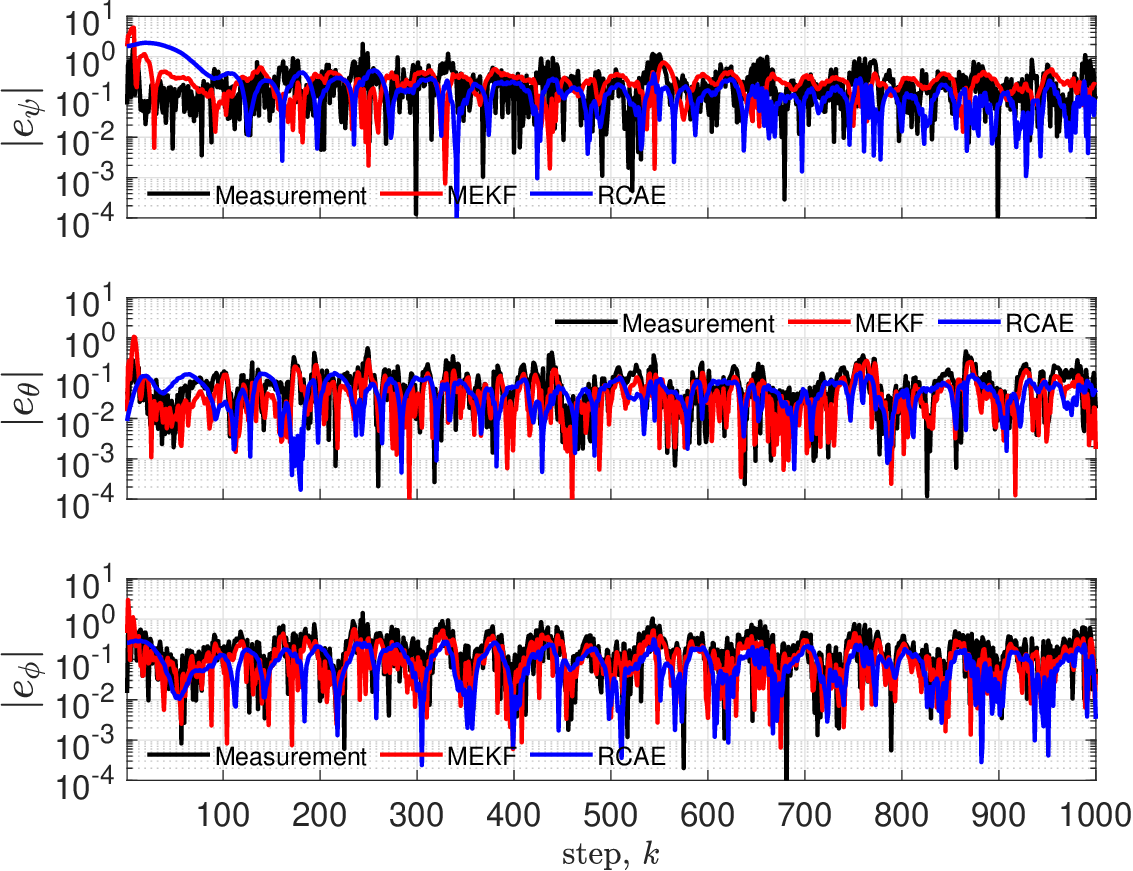}
    \caption{
    3-2-1 Euler angle errors obtained with RCAE and MEKF. 
    % Shows the absolute value of RCAF Euler angle estimation error, absolute value of the MEKF Euler angle estimation error, and the
    % absolute value of the Euler angle measurement error on a logarithmic scale.
    }
    \label{fig:Euler_Angles_RCAF_MEKF_Errors_experiment}
\end{figure}

\section{Conclusions and Future Work}
\label{sec:conclusion}

% This paper presented a novel multiplicative attitude estimator based on retrospective cost optimization. 
% The estimator, called the Retrospective Cost Attitude Estimator (RCAE), is developed for the attitude's $3\times 3$ orientation matrix parameterization.
% Unlike the Kalman filter-based estimation techniques, the RCAE uses an RLS-based optimization algorithm that uses only the measured data and does not require the computation of Jacobians and propagation of covariance matrices.
% % 
% The proposed estimator is validated in a numerical experiment, where it successfully rejects an unknown gyro bias without explicitly estimating it.
% % 
% The future work is focused on
% 1) extension of RCAE to quaternion parameterization, 
% 2) experimental implementation and verification of the retrospective cost attitude estimator in real-world applications, and  
% 3) performance comparison of RCAE with Kalman-filter-based techniques.   

This paper presented a novel learning-based attitude estimator with multiplicative correction, leveraging retrospective cost optimization. The estimator, termed the Retrospective Cost Attitude Estimator (RCAE), is specifically developed for the $3\times 3$ orientation matrix parameterization of attitude. Unlike traditional Kalman filter-based estimation techniques, RCAE employs a Recursive Least Squares (RLS)-based optimization algorithm that uses only the measured data to learn the appropriate estimator gains, eliminating the need for Jacobian computations and covariance matrix propagation. The efficacy of the proposed estimator is validated through comprehensive numerical experiments and a physical experimental setup. RCAE demonstrates robust performance, accurately estimating attitude while effectively rejecting unknown gyro bias without the need for explicit estimation. 

Future work will focus on extending RCAE to quaternion parameterization,
% enhancing its applicability to different attitude representations, 
and implementing and verifying the retrospective cost attitude estimator in real-world applications such as UAV control and navigation, further demonstrating its practical utility and robustness in dynamic environments.
% By focusing on a data-driven, learning-based approach, RCAE offers a promising direction for attitude estimation, leveraging the power of modern machine learning techniques to improve estimation accuracy and computational efficiency.

% directly comparing the performance of the proposed filter with the multiplicative Kalman-filter-based methods. 

% In this work, we present a multiplicative adaptive approach for estimating attitude, which can be utilized directly with $\rm{SO}(3)$ orientation matrices or other attitude representations. This method adeptly handles the presence of unknown gyro bias and adeptly sidesteps the complexities linked to issues such as discontinuities, uniqueness problems, and the computation of extensive covariance matrices.

\printbibliography

\begin{appendices}
\label{appnd:AM}
\section{Attitude Measurement}

% \appendix
% \section{Attitude Measurement}

The orientation matrix corresponding to the attitude of $\SB$ can be constructed using direct onboard measurements.
Several methods exist to construct the orientation matrix from measured data. 
For example, vision-based sensors can be used to compute an orientation measurement as described in \cite{shabayek2012vision, thurrowgood2009vision, kessler2010vision}, or an accelerometer and the magnetometer measurements from an IMU can be used to compute an orientation measurement as described below.

We use the following fact to compute the orientation using the measurements from an IMU.

\begin{fact}
    \label{fact:coordinates}
    Let $\rm F_A$ and $\rm F_B$ be two frames. 
    Let $\vect x$ be a vector. 
    Let $\resolve{\vect x}{A} $ and $\resolve{\vect x}{B} $ denote the coordinates of $\vect x$ in frames $\rm F_A$ and $\rm F_B,$ respectively. 
    Then, 
    \begin{align}
        \resolve{\vect x}{B} 
            =
                \SO_{\rm B/A} \resolve{\vect x}{A},
        \label{eq:coordinates}
    \end{align}
    where $\SO_{\rm B/A}$ is the orientation matrix. 
    Furthermore, it follows from \eqref{eq:coordinates} that
\begin{align}
    \SO_{\rm B/A}
        % &=
        %     \matl 
        %         \resolve{\ihat B}{A} &
        %         \resolve{\jhat B}{A} &
        %         \resolve{\khat B}{A} 
        %     \matr^\rmT
        % \\
        &=
            \matl 
                \resolve{\ihat A}{B} &
                \resolve{\jhat A}{B} &
                \resolve{\khat A}{B} 
            \matr.
    \label{eq:OBA_def}
\end{align}
\end{fact}

% Let $\rm F_A$ and $\rm F_B$ be two frames. 

% The attitude of $\SB,$ parameterized by the orientation matrix $\SO_{\rm B/A},$ can be computed using an IMU measurement, which is assumed to be rigidly fixed to $\SB$ and whose axes are aligned with $\rm F_B.$ 
% 
Let $\rm F_A$ be defined such that $\khat A$ is aligned with the direction of gravity, that is, $\vect g = g \khat A,$ and the magnetic field is in the $\ihat A-\khat A$ plane. 
Let $a \in \BBR^3$ denote the acceleration measurement. 
Assuming that the body $\SB$ is not accelerating, the acceleration measurement $a \approx \resolve{\vect g}{B},$ which implies that 
\begin{align}
\label{eq:k_A_IMU_Measurement}
    \resolve{\khat A}{B} 
        &=
        %     \frac{\resolve{\vect a}{B}}{ \|a \|}
        % =
            \frac{a}{ \|a \|}.
\end{align}
Next, Let $m \in \BBR^3$ denote the magnetic field measurements.
Since $\vect m$ is assumed to be in the $\ihat {A} - \khat {A}$ plane and, for all $\alpha, \beta \in \BBR,$
$\khat{A} \times (\alpha \ihat{A} + \beta \khat{A}) = \alpha \jhat{A},$ it follows that 
\begin{align}
\label{eq:j_A_IMU_Measurement}
    \jhat{A} \resolvedin{B}
    % \resolve{\jhat A}{B}
        =
            \khat{A} \resolvedin{B} \times \frac{m }{\|m \|}
        =
            \frac{a}{\|a\|} \times \frac{m }{\|m \|}
\end{align}
and  
\begin{align}
\label{eq:i_A_IMU_Measurement}
    \ihat{A} \resolvedin{B}
        =
            \khat{A} \resolvedin{B} \times \jhat{A} \resolvedin{B}
        =
            \frac{a}{\|a\|} \times
            \left(
                \frac{a}{\|a\|} \times \frac{m }{\|m \|}
            \right).
\end{align}
The orientation matrix $\SO_{\rm B/A}$ is finally given by \eqref{eq:OBA_def} using 
$\ihat{A} \resolvedin{B},$ 
$\jhat{A} \resolvedin{B},$ and
$\khat{A} \resolvedin{B}$ computed above.

\end{appendices}
\end{document}

%% file: PackagesCommands.tex
\usepackage{amsmath} % assumes amsmath package installed
\usepackage{amssymb}  % assumes amsmath package installed
\usepackage{amsfonts}

\newtheorem{proposition}{Proposition}[section]
\newtheorem{fact}{Fact}[section]

\usepackage{graphicx}
\usepackage{epstopdf}
\usepackage{float} 
\usepackage[skip=0pt,font=footnotesize]{caption}
\usepackage[margin = 1in]{geometry}

\usepackage{subcaption}

\usepackage{etaremune}

\input{TikzStuff}

\usepackage{hyperref}
\usepackage{xcolor}
\hypersetup{
    colorlinks,
    linkcolor={blue!100!black},
    citecolor={blue!50!black},
    urlcolor={blue!80!black}
}

\newcommand{\bc}{\begin{center}}
\newcommand{\ec}{\end{center}}
\newcommand{\benum}{\begin{enumerate}}
\newcommand{\eenum}{\end{enumerate}}
\newcommand{\nn}{\nonumber}
\newcommand{\matl}{\left[ \begin{array}}
\newcommand{\matr}{\end{array} \right]}
\newcommand{\matls}{\left[ \begin{smallmatrix}}
\newcommand{\matrs}{\end{smallmatrix} \right]}
\newcommand{\isdef}{\stackrel{\triangle}{=}}

\newcommand{\vect}[1]{\overset{\rightharpoonup}{#1}}

\newcommand{\tr}{{\rm tr}\,}
\newcommand{\rmA}{{\rm A}}
\newcommand{\rmB}{{\rm B}}

\newcommand{\rmT}{{\rm T}}

\newcommand{\rmd}{{\rm d}}

\newcommand{\rmf}{{\rm f}}

\newcommand{\rmi}{{\rm i}}

\newcommand{\rmm}{{\rm m}}

\newcommand{\rmp}{{\rm p}}

\newcommand{\BBR}{{\mathbb R}}

\newcommand{\SB}{{\mathcal B}}

\newcommand{\SN}{{\mathcal N}}
\newcommand{\SO}{{\mathcal O}}

\newcommand{\shiftq}{{\textbf{\textrm{q}}}}

\renewcommand{\matl}{\begin{bmatrix}}
\renewcommand{\matr}{\end{bmatrix} }

\newcommand{\resolve}[2]{\ensuremath{\left.{#1}\right|_{\rm #2}}}

\newcommand{\ihat}[1]{\ensuremath{\hat \imath_{\rm #1} }}
\newcommand{\jhat}[1]{\ensuremath{\hat \jmath_{\rm #1} }}
\newcommand{\khat}[1]{\ensuremath{\hat k_{\rm #1} }}

\newcommand{\resolvedin}[1]{{\left. \right|_{\rm #1}}}

%% file: TikzStuff.tex
\usepackage{tikz}
\usetikzlibrary{calc,patterns,decorations.pathmorphing,decorations.markings}
\usetikzlibrary{shapes,arrows}
\usepackage{verbatim}

\tikzstyle{block} = [draw, fill=blue!20, rectangle, 
    minimum height=3em, minimum width=6em]
\tikzstyle{sum} = [draw, fill=blue!20, circle, node distance=1cm]
\tikzstyle{input} = [coordinate]
\tikzstyle{output} = [coordinate]
\tikzstyle{pinstyle} = [pin edge={to-,thin,black}]

\usetikzlibrary{positioning}
\usetikzlibrary{math}

\usepackage{tikz}
\usetikzlibrary{shapes,arrows,calc,positioning}
\tikzstyle{bigblock} = [draw, fill=blue!20, rectangle, 
    minimum height=6em, minimum width=8em]
\tikzstyle{medblock} = [draw, fill=blue!20, rectangle, 
    minimum height=4em, minimum width=4em]    
\tikzstyle{mux} = [draw, fill=black!20, rectangle, 
    minimum height=5em, minimum width=0.1em]    
\tikzstyle{smallblock} = [draw, fill=blue!20, rectangle, 
    minimum height=3em, minimum width=4em]
\tikzstyle{sum} = [draw, fill=blue!20, circle, node distance=1cm]
\tikzstyle{signal} = [coordinate]
\tikzstyle{pinstyle} = [pin edge={to-,thin,black}]
\tikzstyle{block} = [draw, fill=blue!20, rectangle, 
    minimum height=3em, minimum width=6em]
\tikzstyle{blockS} = [draw, fill=blue!20, rectangle, 
    minimum height=3em, minimum width=4em]    
\tikzstyle{input} = [coordinate]
\tikzstyle{output} = [coordinate]
\usetikzlibrary{matrix}

\tikzset{add/.style n args={4}{
    minimum width=1mm,
    path picture={
        \draw[black, thick] 
            (path picture bounding box.south east) -- (path picture bounding box.north west)
            (path picture bounding box.south west) -- (path picture bounding box.north east);
        }
    }
}

\tikzset{Frame_into/.pic={
        code={\tikzset{scale=1}
        \tikzmath
            {
                \l  = 1;
                \Rc = 0.15;
            } 
        \draw [thick,->] (0, 0) -- +(\l, 0);
        \draw [thick,->] (0, 0) -- +(0, \l);
        \draw [thick, fill=white] 
			    (0,0) circle [radius=\Rc];
	    \draw [thick] ({\Rc*cos(45)}, {0\Rc*sin(45)}) 
	               -- ({\Rc*cos(225)}, {0\Rc*sin(225)});
        \draw [thick] ({\Rc*cos(135)}, {0\Rc*sin(135)}) 
	               -- ({\Rc*cos(315)}, {0\Rc*sin(315)});
  }}
}

\tikzset{Frame_outof/.pic={
        code={\tikzset{scale=1}
        \tikzmath
            {
                \l  = 1;
                \Rc = 0.15;
            } 
        \draw [thick,->] (0, 0) -- +(\l, 0);
        \draw [thick,->] (0, 0) -- +(0, \l);
        \draw [thick, fill=white] 
			    (0,0) circle [radius=\Rc];
	    \draw [thick, fill=black] 
			    (0,0) circle [radius={0.2*\Rc}];
  }}
}